\documentclass{amsart}
\usepackage[utf8]{inputenc}
\usepackage{amsmath, amsfonts, amsthm, amssymb,epsfig}
\usepackage{bibentry}
\usepackage{ytableau}
\usepackage{array}
\usepackage{pdflscape}
\usepackage{multicol}
\usepackage[shortlabels]{enumitem}
\usepackage{ytableau}
\usepackage{diagbox}
\usepackage{slashbox}
\usepackage{mathrsfs}

\DeclareFontFamily{U}{mathc}{}
\DeclareFontShape{U}{mathc}{m}{it}%
{<->s*[1.03] mathc10}{}

\DeclareMathAlphabet{\mathscr}{U}{mathc}{m}{it}

\theoremstyle{definition}

\newtheorem*{theorem*}{Theorem}
\newtheorem*{definition*}{Definition}

\newtheorem{theorem}{Theorem}[section]

\newtheorem{prop}[theorem]{Proposition}

\newtheorem{ex}[theorem]{Example}
\newtheorem{rmk}[theorem]{Remark}

\newcommand{\nn}{\mathbb{N}}
\newcommand{\Span}{\mathrm{Span}}

\begin{document}

\title[]{A recursion formula for Branching from $\mathfrak{sl}_n$ to $\mathfrak{sl}_2$ Subalgebras
}

\author{Korkeat Korkeathikhun}
\address{Korkeat Korkeathikhun \\ 1) Department of Mathematics, Faculty of Science, National University of Singapore, 119076, Singapore \\
2) Department of Mathematics and Computer Science, Faculty of Science, Chulalongkorn University, Bangkok 10330, Thailand}
\email{\tt korkeat.k@chula.ac.th, korkeat.k@gmail.com}

\author{Borworn Khuhirun}
\address{Borworn Khuhirun\\ Department of Mathematics and Statistics, Faculty of Science and Technology, Thammasat University, Rangsit Campus, Pathum Thani 12121, Thailand} 
\email{\tt borwornk@mathstat.sci.tu.ac.th}

\author{Songpon Sriwongsa$^*$}
\thanks{*Corresponding Author}
\address{Songpon Sriwongsa\\ Department of Mathematics, Faculty of Science, King Mongkut's University of Technology Thonburi (KMUTT), Bangkok 10140, Thailand} 
\email{\tt songpon.sri@kmutt.ac.th, songponsriwongsa@gmail.com}

\author{Keng Wiboonton}
\address{Keng Wiboonton\\ Department of Mathematics and Computer Science, Faculty of Science, Chulalongkorn University, Bangkok 10330, Thailand} 
\email{\tt keng.w@chula.ac.th, kwiboonton@gmail.com}



\date{}

\begin{abstract}
    For any representation of a complex simple Lie algebra $\mathfrak{sl}_n$, one problem of branching rules to $\mathfrak{sl}_2$-subalgebra is to determine the multiplicity of each irreducible component. In this paper, we derive a recursion formula of such multiplicities by restricting a certain tensor representation in two ways, in which the Pieri's rule is involved. We also investigate branching rules for fundamental representations as they are initial conditions of the recursion formula.
\end{abstract}

\maketitle

\section{Introduction}

\bigskip

In this paper, the ground field is assumed to be the set of all complex numbers $\mathbb{C}$. Let $\mathfrak{g}$ be a simple Lie algebra and $V$ a finite dimensional representation of $\mathfrak{g}$. Let $\mathfrak{h}$ be a semisimple subalgebra of $\mathfrak{g}$. Then $V$ can also be regarded as a $\mathfrak{h}$-representation. This consideration is called the \textit{restriction of $V$ to $\mathfrak{h}$}, denoted by $\mathrm{Res}_{\mathfrak{h}}^{\mathfrak{g}} V$.
If $V$ is an irreducible representation of $\mathfrak{g}$, then $\mathrm{Res}_{\mathfrak{h}}^{\mathfrak{g}} V$ may be decomposed into the direct sum of irreducible representations of $\mathfrak{h}$. One can naturally ask what the multiplicity of each irreducible component is. An explicit description of these numbers is called a \textit{branching rule} or \textit{branching law}. If $V$ and $W$ are isomorphic as $\mathfrak{g}$-representations, we use the notation $V\cong_{\mathfrak{g}} W$.

    The study of branching rules is complicated in general. We may start with a simple case, $\mathfrak{g}=\mathfrak{sl}_n := \mathfrak{sl}_n(\mathbb{C})$, the set of all $n\times n$ traceless matrices, together with its irreducible representation $V$ of finite dimension and $\mathrm{Res}_{\mathfrak{s}}^{\mathfrak{sl}_n}V$, the restriction of $V$ to a subalgebra $\mathfrak{s}$ isomorphic to $\mathfrak{sl}_2$, which is known as an $\mathfrak{sl}_2$-subalgebra.
     
     Denote the irreducible $\mathfrak{sl}_2$-representation of highest weight $i$ by $F_i$. Once $V$ is restricted to $\mathfrak{s}$, we have a decomposition
    $$V \ \cong_{\mathfrak{s}} \ m_0 F_0\oplus m_1 F_1\oplus m_2 F_2\oplus \cdots.$$
    Each $m_j$ is called the {\it multiplicity} of $F_j$ in the $\mathfrak{s}$-representation $V$. We sometimes denote this multiplicity by $\mathrm{mult}(F_j: V)$. If there is no confusion, we simply write $m_j$. Since the $m_j$'s are multiplicities in the decomposition of the $\mathfrak{sl}_2$-representation, we have a well-known description of each $m_j$ as the following (see for example in \cite[Corollary 5.2.3]{He}):
    \begin{equation}\label{description of muktiplicities} m_j \ = \ \dim V_j -\dim V_{j+2},\end{equation}
    where each $V_j$ is the weight space of the weight $j$ in the $\mathfrak{s}$-representation $V$. Then, we also have that the number of irreducible summands of $V$ is equal to $$m_0+m_1+\cdots \ = \ \dim V_0 + \dim V_1.$$ Therefore, the number of irreducible summands of $V$ is equal to the sum of the multiplicities of the weight $0$ and the weight $1$. Since $V$ is finite dimensional, the \textit{highest component} $h :=\max \{j\mid m_j\neq 0\}$ exists. If $V$ is an irreducible representation of $\mathfrak{sl}_n$ with the highest weight $\lambda$, it is mentioned in \cite[Page 23]{Ka} that, as branching to a principal $\mathfrak{sl}_2$ subalgebra,
    
    \begin{align}\label{highest comp}
     h = \ \sum_{\text{positive roots }\alpha} \lambda(H_{\alpha}),    
    \end{align}
    where $H_{\alpha}$ is a semisimple element of an $\mathfrak{sl}_2$-triplet $\{X_{\alpha}, Y_{\alpha}, H_{\alpha} \}$ corresponding to a positive root $\alpha$. 
     Analogously, we can define the \textit{lowest component} $\ell :=\min \{j\mid m_j\neq 0\}$. Using these terminologies, we can then write 
    $$V \ \cong_{\mathfrak{s}} \ m_{\ell} F_{\ell}\oplus \cdots \oplus m_h F_h.$$
    One study of Branching algebra is about describing $\ell$ and its multiplicity. In \cite{LW}, an upper bound of $\ell$ is given, namely $\ell<n$. Unlike the highest component $h$, it is difficult to determine what $\ell$ exactly is. It is even harder to determine the multiplicity of $F_j$ for any $j$. This is a challenging problem in representation theory. Many scholars have been trying to find a simple and optimal way to compute those multiplicities. By providing a concrete and computationally efficient tool, this work enhances the applicability of branching rules in both theoretical and applied contexts.

    In this work, we propose an alternative way to compute the multiplicities $m_j$'s through the following main theorem where the proof is provided in Section \ref{section: recursion}.

\begin{theorem}\label{thm: general recursion formula}
    Let $\mathfrak{s}$ be an $\mathfrak{sl}_2$-subalgebra of $\mathfrak{sl}_n$. For each dominant integral weight $\lambda\in \Lambda^+$, let $m_j(\lambda)=\mathrm{mult}(F_{j}: \mathrm{Res}_{\mathfrak{s}}^{\mathfrak{sl}_n}L(\lambda))$. Then for $d\in \mathbb{Z}_{\geq 0}$ and $k=1, 2, \dots, n-1$,
    $$m_d(\lambda+\omega_k)=\sum_{(j,j')} m_j(\lambda)m_{j'}(\omega_k) \ -  \sum_{\mu\in P(\lambda, k)-\{\lambda+ \omega_k\}}m_d(\mu),$$
    where the first summation ranges over $(j,j')\in \mathbb{Z}_{\geq 0}\times \mathbb{Z}_{\geq 0}$ such that $|j-j'|\leq d\leq j+j'$ and $2|(j+j'-d)$ and whose summands are all zeros but finitely many. In particular,
    $$m_0(\lambda+\omega_k) \ = \ \sum_{j=0}^\infty m_j(\lambda)m_j(\omega_k) \ -  \sum_{\mu\in P(\lambda, k)-\{\lambda+ \omega_k\}}m_0(\mu).$$
    Here $P(\lambda, k)$ is the set of dominant integral weights obtained by adding $k$ boxes into the Young diagram of $\lambda$ such that no more than one is in the same row.
\end{theorem}

As branching rules for fundamental representations are initial conditions of this formula, we study them in section \ref{section: fund rep}. In section \ref{section: example coro}, we provide some examples and a consequence of our recursion formula.

\bigskip

\section{Recursion Formula for Multiplicity of $F_d$}\label{section: recursion}
    
    \bigskip
    
   We mainly follow the notations and conventions of representations of $\mathfrak{sl}_n$ from \cite{Fu, GW}, and we recall them here. Let $E_{ij}$ be the matrix with $1$ in the $i$ th row of the $j$ th column and $0$’s everywhere else. Let $H_i=E_{ii}$ and 
    $$\mathfrak{h}=\{a_1 H_1+ a_2 H_2+\cdots +a_n H_n \mid a_1, \dots, a_n \in \mathbb{C} \text{ and } a_1 + a_2 +\cdots +a_n=0\}.$$
    Then $\mathfrak{h}$ is the set of all traceless diagonal matrices which is a Cartan subalgebra of $\mathfrak{sl}_n$. 
    For $i = 1, \ldots, n$, let $\varepsilon_i : \mathfrak{h} \rightarrow \mathbb{C}$ be a linear functional defined by  
    $$
    \varepsilon_i(\sum_j a_jH_j) = a_i.
    $$
    Note that $\varepsilon_1 + \dots +\varepsilon_n = 0$.
     For $i = 1,\dots , n-1$, 
    let
    $$\alpha_i = \varepsilon_i - \varepsilon_{i+1},$$
    $$\omega_i = \varepsilon_1 + \varepsilon_2 + \cdots + \varepsilon_i.
    $$
    Then $\alpha_i$'s and $\omega_i$'s are simple roots and fundamental weights of $\mathfrak{sl}_n$, respectively. Let $\Lambda$ and $\Lambda^+$ be the weight lattice
    and the set of all dominant integral weights of $\mathfrak{sl}_n$ which are determined by
    $$
    \Lambda = \Span_\mathbb{Z}\{\omega_1, \omega_2,\dots, \omega_n\} \text{ and }
    \Lambda^+ = \Span_{\mathbb{Z}_{\geq 0}}\{\omega_1, \omega_2,\dots, \omega_n\},
    $$
    respectively. One can associate each $$\lambda=a_1\omega_1+ a_2\omega_2+ \cdots + a_{n-1}\omega_{n-1}\in \Lambda^+ (a_i\in \mathbb{Z}_{\geq 0})$$ with the partition $(\lambda_1 \geq  \lambda_2\geq \dots\geq \lambda_{n-1} \geq 0)$ where
    \begin{align*}
       \lambda_1 &= a_1+ a_2+ \cdots + a_{n-1}\\
       \lambda_2 &= a_2+ \cdots + a_{n-1}\\ 
       &\  \vdots\\
      \lambda_{n-1} &= a_{n-1}.
    \end{align*}
    In other words, $\lambda_i$ equals the coefficient of $\varepsilon_i$. Define a total order $\prec$ on $\Lambda^+$ by $\lambda \prec \mu$ if $(\lambda_1 \geq  \lambda_2\geq \dots\geq \lambda_{n-1} \geq 0)$ less than $(\mu_1 \geq  \mu_2\geq \dots\geq \mu_{n-1} \geq 0)$ with respect to the lexicographic order. For each $\lambda\in \Lambda^+$, let $L(\lambda)$ be the irreducible representation corresponding to $\lambda$. 
    
    We recall the Pieri's rule. Let $\lambda\in \Lambda^+$ and $k=1,2, \dots, n-1$. As representations of $\mathfrak{sl}_n$,
    $$L(\lambda)\otimes L(\omega_k) \ \cong \ \bigoplus_{\mu \in P(\lambda, k)} L(\mu)$$
    where $P(\lambda, k)$ is the set of dominant integral weights obtained by adding $k$ boxes into the Young diagram of $\lambda$ such that no more than one is in the same row. 

\bigskip

\begin{ex} 
  $\lambda=2\omega_2 +\omega_3$ in $\mathfrak{sl}_4$ has the Young diagram
    \medskip
    \begin{center}
       \ytableausetup{centertableaux}
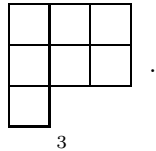
        \begin{ytableau}
        \none & & & \\
        \none & & &\\
        \none &
    \end{ytableau}\ \ .
    \end{center}
    \medskip
    To determine $P(\lambda, 2)$, we add 2 boxes filled by $\times$ into the Young diagram of $\lambda$ in such a way that no more than one is in the same row and the results are in the Young diagram shape. Then $P(\lambda, 2)$ consists of the following weights:
    \medskip
    \begin{align*}
     \ytableausetup{centertableaux}
        \begin{ytableau}
            \none & & & & \times\\
            \none & & & & \times \\
            \none &
        \end{ytableau}
        &\,\, \ = \ 3\omega_2 + \omega_3,
        \\
        \\
        \begin{ytableau}
            \none & & & & \times\\
            \none & & & \\
            \none & & \times
        \end{ytableau}
        &\,\,\ = \ \omega_1+ \omega_2 + 2\omega_3,
        \\
        \\
        \begin{ytableau}
            \none & & & &\times \\
            \none & & & \\
            \none & \\
            \none & \times
        \end{ytableau}
        &\,\,=
        \begin{ytableau}
            \none & & & \\
            \none & & \\
        \end{ytableau}
        \,\, \ = \ \omega_1 + 2\omega_2,
        \\
        \\
        \begin{ytableau}
            \none & & & \\
            \none & & & \\
            \none & &\times\\
            \none & \times
        \end{ytableau}
         &\,\,=
         \begin{ytableau}
            \none & & \\
            \none & & \\
            \none & \\
        \end{ytableau}
        \,\,\ = \ \omega_2 + \omega_3.     
    \end{align*} 
    
    \medskip
    
    The last two Young diagrams are subject to the condition that $\varepsilon_1 + \varepsilon_2 + \varepsilon_3 + \varepsilon_4 = 0$.
\end{ex}    
 
 \medskip
 
\begin{prop}\label{prop: Pieri}
    $\lambda+ \omega_k\in P(\lambda, k)$ and $\lambda+ \omega_k\succeq \mu$ for any $\mu\in P(\lambda, k)$.
\end{prop}
\begin{proof}
    Note that $\lambda+ \omega_k$ is obtained by adding $k$ boxes into the first $k$ rows of the Young diagram of $\lambda$. Young diagram obtained by adding $k$ boxes in other ways yields partitions which are lower than $\lambda+ \omega_k$ in the lexicographic order.
\end{proof}


    Now, we are ready to prove the main theorem, which gives a method to recursively compute the multiplicity of $F_d$ in $\mathrm{Res}_{\mathfrak{s}}^{\mathfrak{sl}_n}L(\lambda)$ for any $\lambda\in \Lambda^+$.
    
    \bigskip
    
\medskip
\noindent\textbf{Proof of Theorem \ref{thm: general recursion formula}}:
    Consider $L(\lambda)\otimes L(\omega_k)$ as an irreducible representation of $\mathfrak{sl}_n\oplus \mathfrak{sl}_n$. Apply Pieri's rule when we restrict it to $\mathfrak{sl}_n$ so that
    $$\mathrm{Res}_{\mathfrak{sl}_n}^{\mathfrak{sl}_n\oplus \mathfrak{sl}_n} (L(\lambda)\otimes L(\omega_k)) \ \cong \ \bigoplus_{\mu \in P(\lambda, k)} L(\mu).$$
    Then we restrict it further to the principal subalgebra $\mathfrak{s}$:
    $$\mathrm{Res}_{\mathfrak{s}}^{\mathfrak{sl}_n\oplus \mathfrak{sl}_n}(L(\lambda)\otimes L(\omega_k)) \ \cong \ \bigoplus_{\mu \in P(\lambda, k)}\bigoplus_k m_k(\mu) F_k.$$
    The multiplicity of $F_d$ equals $\sum_{\mu\in P(\lambda, k)} m_d(\mu)$.
    
    On the other hands, restricting $L(\lambda)\otimes L(\omega_k)$ to $\mathfrak{s}\oplus \mathfrak{s}$, we have
    \begin{align*}
        \mathrm{Res}_{\mathfrak{s}\oplus \mathfrak{s}}^{\mathfrak{sl}_n\oplus \mathfrak{sl}_n}(L(\lambda)\otimes L(\omega_k))
        & \ \cong \ \mathrm{Res}_{\mathfrak{s}}^{\mathfrak{sl}_n}L(\lambda)\otimes \mathrm{Res}_{\mathfrak{s}}^{\mathfrak{sl}_n}L(\omega_k)\\
        & \ \cong \ \left(\bigoplus_j m_j(\lambda) F_j\right) \otimes \left(\bigoplus_{j'} m_{j'}(\omega_k) F_{j'}\right)\\
        & \ \cong \ \bigoplus_j \bigoplus_{j'} m_j(\lambda)m_{j'}(\omega_k) F_j\otimes F_{j'}.
    \end{align*}
    Then we restrict it further to $\mathfrak{s}$:
    $$\mathrm{Res}_{\mathfrak{s}}^{\mathfrak{sl}_n\oplus \mathfrak{sl}_n}L(\lambda)\otimes L(\omega_k) \ \cong \ \bigoplus_j \bigoplus_{j'} \bigoplus_{k=|j-j'|}^{j+j'} m_j(\lambda)m_{j'}(\omega_k) F_k.$$
    On the right hand side, we observe that $F_d$ occurs if and only if $|j-j'|\leq d\leq j+j'$ and $2|(j+j'-d)$. Then the multiplicity of $F_d$ equals $
    \sum_{(j,j')} m_j(\lambda)m_{j'}(\omega_k)$ where the summation ranges over $(j,j')\in \mathbb{Z}_{\geq 0}\times \mathbb{Z}_{\geq 0}$ such that $|j-j'|\leq d\leq j+j'$ and $2|(j+j'-d)$. The result follows from Proposition \ref{prop:  Pieri}.  \qed

\section{Branching rules of Fundamental $\mathfrak{sl}_n$ Representations}\label{section: fund rep}

\medskip

    In this section, we compute the multiplicity of $\mathrm{Res}_{\mathfrak{s}}^{\mathfrak{sl}_n} L(\omega_k)$ as they are the initial values for the recursion formula in Theorem \ref{thm: general recursion formula}. We divide this into two cases: principal and non-principal $\mathfrak{sl}_2$ subalgebras.

    Recall that any $\mathfrak{sl}_2$ subalgebra is generated by a triple $\{H,X,Y\}$ with $X$ nilpotent. By \cite[Theorem 3.2.10]{CM}, there exists a one-to-one correspondence between $\mathrm{SL}_n$-conjugacy classes of $\mathfrak{sl}_2$ subalgebras of $\mathfrak{sl}_n$ and nonzero nilpotent $\mathrm{SL}_n$-orbits in $\mathfrak{sl}_n$.
    Through Jordan normal forms, there is a one-to-one correspondence between nilpotent $\mathrm{SL}_n$-orbits in $\mathfrak{sl}_n$ and partitions of $n$. 

    Following the terminologies in \cite{CM}, a \textit{regular element} $x$ of $\mathfrak{g}$ is an element such that $\dim \mathfrak{g}^x = \mathrm{rank} \ \mathfrak{g}$ where $\mathfrak{g}^x$ is the centralizer of $x$ in $\mathfrak{g}$. An element $e\in\mathfrak{g}$ is {\it regular nilpotent} if $e$ is regular and $\mathrm{ad}_e$ is nilpotent. Any element of $\mathfrak{g}$ that is conjugate to $e$ is also regular nilpotent. In fact, all regular nilpotent elements are conjugate and form the principal (regular) orbit in $\mathfrak{g}$. A \textit{principal $\mathfrak{sl}_2$-subalgebra} is a subalgebra $\mathfrak{s}$ of $\mathfrak{g}$ such that $\mathfrak{s}\cong \mathfrak{sl}_2$ and $\mathfrak{s}$ contains a regular nilpotent element \cite{Ko1}. In the case $\mathfrak{g}=\mathfrak{sl}_n$, $\mathfrak{s}$ is principal if and only if it contains a nilpotent element $X$ whose Jordan block has only one block of size $n$.     
    
    If $X$ is a nilpotent element of $\mathfrak{sl}_n$ whose Jordan blocks have sizes $d_1\geq d_2\geq \cdots \geq d_m\geq 1$, we say that $X$ is nilpotent of type $[d_1,d_2,...,d_m]$. For each partition $[d_1,d_2,...,d_m]$ of $n$, let $\mathfrak{s}_{[d_1,d_2,...,d_m]}$ be the $\mathfrak{sl}_2$ subalgebra generated by 
    \[H_{[d_1,d_2,...,d_m]} \ = \ 
    \begin{pmatrix}
    H_{d_1} & & &   \\
    & H_{d_2} & &   \\
    & & \ddots &  \\
    & & &  H_{d_m}
    \end{pmatrix}
    \]\\
    \[X_{[d_1,d_2,...,d_m]} \ = \ 
    \begin{pmatrix}
    X_{d_1} & & &\\
    & X_{d_2} & &  \\
    & & \ddots & \\
    & & & X_{d_m}
    \end{pmatrix}
    \]\\
    and
     \[Y_{[d_1,d_2,...,d_m]} \ = \ 
    \begin{pmatrix}
    Y_{d_1} & & &\\
    & Y_{d_2} & & \\
    & & \ddots & \\
    & & & Y_{d_m}
    \end{pmatrix}
    \]
    where for $r\in \mathbb{N}$,
    \[H_{r} \ = \ 
    \begin{pmatrix}
    r-1 & & & & \\
    & r-3 & & & \\
    & & \ddots & &\\
    & & & -(r-3) &\\
    & & & & -(r-1)
    \end{pmatrix},
    \]
    $$
    X_r \ = \ \begin{pmatrix}
    0 & 1 &  &  &  & \\
     & 0 & 1 &  &  & \\
     &  & 0 & \ddots &  & \\
     &  &  & \ddots & 1 & \\
     &  &  &  & 0 & 1\\
     &  &  &  &  & 0\\
    \end{pmatrix}, \hspace{0.5 cm}
    Y_r \ = \ \begin{pmatrix}
    0 &  &  &  &  & \\
    \lambda_1 & 0 &  &  &  & \\
     & \lambda_2 & 0 &  &  & \\
     &  & \lambda_3 & \ddots &  & \\
     &  &  & \ddots & 0 & \\
     &  &  &  & \lambda_{n-1} & 0\\
    \end{pmatrix}
    $$
    
    \medskip
    \noindent and $\lambda_i=i(n-i)$ for $i=1,\dots, n-1$. Then $\mathfrak{s}_{[d_1,d_2,...,d_m]}$ corresponds to the nilpotent $\mathrm{SL}_n$-orbits of $X_{[d_1,d_2,...,d_m]}$ and the partition $[d_1,d_2,...,d_m]$ of $n$.

    For any irreducible $\mathfrak{sl}_n$-representation $V$, denote the multiset of weights in $V$ restricting to $\mathfrak{s}$ (resp. $\mathfrak{s'}$) by $\Lambda_{\mathfrak{s}}$ (resp. $\Lambda'_{\mathfrak{s}}$). If $\mathfrak{s}$ and $\mathfrak{s}'$ correspond to the same partition (equivalently, they are $\mathrm{SL}_n$-conjugate), then $\Lambda_{\mathfrak{s}}=\Lambda'_{\mathfrak{s}}$.
    Consequently, when we restrict an irreducible $\mathfrak{sl}_n$-representation to any $\mathfrak{sl}_2$-subalgebra, it suffices to restrict to $\mathfrak{s}_{[d_1,d_2,...,d_m]}$. Recall that $L(\omega_k)\cong_{\mathfrak{sl}_n} \wedge^k (\mathbb{C}^n)$ for $k=1, 2,\dots n-1$.

    
\subsection{Principal subalgebras}\label{subsection: principal}
    In this subsection,  we let $\mathfrak{s}=\mathfrak{s}_{[n]}$ be a principal $\mathfrak{sl}_2$-subalgebra generated by $H=H_n, X=X_n, Y=Y_n$. We firstly study the restriction of $L(\omega_1)\cong_{\mathfrak{sl}_n} \mathbb{C}^n$ to $\mathfrak{s}$. Then $H$ is regular, all eigenvalues of $\mathrm{ad}_H$ are even integers, and $H$ is diagonalizable with $n$ distinct eigenvalues, say $\lambda_1,\lambda_2,...,\lambda_n$. Indeed, all eigenvalues of $\mathrm{ad}_H$ are of the form $\lambda_i - \lambda_j$. Hence $\lambda_i$'s are all even or all odd. It is well-known that in decomposing the $\mathfrak{s}$-representation $\mathbb{C}^n$, the number of irreducible summands is the sum of dimensions of 0-weight space and 1-weight space, which is equal to the sum of multiplicities of eigenvalues 0 and 1 of $H$ acting on $\mathbb{C}^n$. Since $\lambda_i$ are distinct and all even or all odd, a number of irreducible summands is exactly one. Therefore, $\mathrm{Res}_{\mathfrak{s}}^{\mathfrak{sl}_n} L(\omega_1) \cong F_{n-1}$.
    
	Now, we consider $\mathrm{Res}_{\mathfrak{s}}^{\mathfrak{sl}_n} L(\omega_k), k=1,\dots, n-1$. For $d,n,k\in \mathbb{N}$, define
	\begin{align*}
		P_k (d)&:= \{(a_1,\dots,a_k)\mid a_i\in \nn, a_i>a_{i+1} \text{ and } a_1+\cdots+ a_k=d \},\\
		P_k^n(d)&:= \{(a_1,\dots,a_k)\mid a_i\in \{1,2,\dots,n\}, a_i>a_{i+1} \text{ and } a_1+\cdots+ a_k=d \},
	\end{align*}
	and define
	$$p_k (d):= \mid P_k (d)\mid, \hspace{0.5 cm} p_k^n(d):= \mid P_k^n(d)\mid.$$ 
	
	
	
	\medskip 
	
	\begin{prop}\label{prop: mult}
		For each $k=1,2,...,n-1$ and nonnegative integer $j$,
		$$\mathrm{mult}(F_j:\mathrm{Res}_{\mathfrak{s}}^{\mathfrak{sl}_n}L(\omega_k))=p_k^n\left(\frac{kn-j+k}{2}\right)-p_k^n\left(\frac{kn-j+k}{2}-1\right).$$
	\end{prop}
	\begin{proof}
        Consider $H_n$ acting on $\wedge^k (\mathbb{C}^n)$ with the natural basis $$\{e_{i_1}\wedge e_{i_2}\wedge \cdots \wedge e_{i_k}\mid 1\leq i_1 < i_2 <\cdots < i_k\leq n\}.$$ Let $V_j$ be the weight space of weight $j$ in $\wedge^k (\mathbb{C}^n)$ as a representation of $\mathfrak{s}$.
		It is easy to check that 
		$$\{e_{i_1}\wedge e_{i_2}\wedge \cdots \wedge e_{i_k} \mid 1\leq i_1 < i_2 <\cdots < i_k\leq n \text{ and } i_1 +\cdots +i_k=\frac{kn-j+k}{2}\}.$$
		is a basis of $V_j$. Then
		$$\mathrm{mult}(F_j:L(\omega_k))=\dim V_j - \dim V_{j+2}
		=p_k^n\left(\frac{kn-j+k}{2}\right)-p_k^n\left(\frac{kn-j+k}{2}-1\right)$$
		as desired.
	\end{proof}

    We already described above that $\mathrm{Res}_{\mathfrak{s}}^{\mathfrak{sl}_n} L(\omega_1) \cong F_d$ where $d=n-1$. It is known that $F_d \cong_{\mathfrak{s}} \mathrm{S}^d(\mathbb{C}^2)$. Then
    $$L(\omega_k)\cong_{\mathfrak{sl}_n} \wedge^k (\mathbb{C}^n) \cong_{\mathfrak{s}} \wedge^k (F_d) \cong_{\mathfrak{s}} \wedge^k(\mathrm{S}^d(\mathbb{C}^2)).$$
    By \cite{LW}, the $q$-binomial coefficient of $\wedge^k(\mathrm{S}^d(\mathbb{C}^2))$ and $\mathrm{S}^k(\mathrm{S}^{n-k}(\mathbb{C}^2))$ are equal:
    $$\binom{d+1}{k}_q=\binom{k+(n-k)}{k}_q.$$
    Hence
    \begin{equation}\label{Res SS}
        \mathrm{Res}_{\mathfrak{s}}^{\mathfrak{sl}_n} L(\omega_k) \cong_{\mathfrak{s}} \wedge^k(\mathrm{S}^d(\mathbb{C}^2)) \cong_{\mathfrak{s}} \mathrm{S}^k(\mathrm{S}^{n-k}(\mathbb{C}^2)).
    \end{equation}
    Therefore, the problem of principal $\mathfrak{sl}_2$-types in fundamental representations has a direct relation with plethysms $\mathrm{S}^k(\mathrm{S}^d(\mathbb{C}^2))$.
    
    One can apply \cite[I.8]{Ma} to describe multiplicities in $\mathrm{Res}_{\mathfrak{s}}^{\mathfrak{sl}_n} L(\omega_k)$ for $k=2, 3$ as follow. Recall that irreducible representation of $\mathrm{SL}_n$ can be parametrized by partitions $\lambda$ with length at most $n$. Their characters $s_\lambda$ are Schur functions on $n$ variables. In particular, as a representation of $\mathrm{SL}_2$,
    $$s_{(\lambda_1,\lambda_2)}(q,q^{-1})=\sum_{j=0}^{\lambda_1-\lambda_2}q^{2\lambda_2-\lambda_1 +2j}$$
    which is the same as the character of $F_{\lambda_1-\lambda_2}.$ Hence
    $$\mathrm{mult}(s_{(\lambda_1,\lambda_2)}: s_{(k)}\circ s_{(n-k)})=\mathrm{mult}(F_{\lambda_1-\lambda_2}: \mathrm{Res}_{\mathfrak{s}}^{\mathfrak{sl}_n}L(\omega_k)).$$
    
    For $k=2$, by \cite[I.8, 9(a)]{Ma},
    $$s_{(2)}\circ s_{(n-2)}=\sum_{i\,\text{even}}s_{(2n-i-4, i)}.$$
    Hence, for each integer $j\geq 0$,
    \[
        \mathrm{mult}(F_{j}: \mathrm{Res}_{\mathfrak{s}}^{\mathfrak{sl}_n}L(\omega_2))=\mathrm{mult}(s_{(\lambda_1,\lambda_2)}: s_{(2)}\circ s_{(n-2)})
        =
        \begin{cases}
        1     & \quad \text{if } \lambda_2 \text{ is even, }\\
        0       & \quad \text{otherwise,}
        \end{cases}
    \]
    where $\lambda_1=\frac{1}{2}(2n+j-4)$ and $\lambda_2=\frac{1}{2}(2n-j-4)$.
    
    For $k=3$, by \cite[I.8, 9(b)]{Ma},
    $$s_{(3)}\circ s_{(n-3)}=\sum_{\mu}\left(\left\lfloor\frac{1}{6}m(\mu)\right\rfloor+\varepsilon(\mu)\right)s_{(\mu)},$$
    where it is summed over partition $\mu=(\mu_1, \mu_2, \mu_3)$ of $3n-9$ with lengths at most 3,
    \[\varepsilon(\mu)=
    \begin{cases}
    1     & \quad \text{if } m(\mu)\equiv \mu_2 \equiv 0 \mod 2,\, \text{or } m(\mu) \equiv 3 \text{ or } 5 \mod 6,\\
    0       & \quad \text{otherwise,}\\
  \end{cases}
    \]
    and $m(\mu)=\min\{\mu_1-\mu_2,\mu_2-\mu_3\}$. Hence, for each integer $j\geq 0$,
    $$\mathrm{mult}(F_{j}: \mathrm{Res}_{\mathfrak{s}}^{\mathfrak{sl}_n}L(\omega_3))=\mathrm{mult}(s_{(\mu_1,\mu_2)}: s_{(3)}\circ s_{(n-3)})=\left\lfloor\frac{1}{6}m(\mu)\right\rfloor+\varepsilon(\mu),$$
    where $\mu_1=\frac{1}{2}(3n+j-9)$ and $\mu_2=\frac{1}{2}(3n-j-9)$.
    
    The plethysms $\mathrm{S}^k(\mathrm{S}^d(\mathbb{C}^r))$ has been extensively studied. For the explicit decomposition when $k=2, 3$ and $4$, we refer the reader to \cite{Ho2}.  The more recent paper \cite{KM} (and the references therein) also contains the results for $k$ up to $5$.
    
   The multiplicities of $F_j$ in $\mathrm{S}^m(\mathrm{S}^d(\mathbb{C}^2))$ can also be described by Cayley-Sylvester formula (see \cite[Theorem 3.3.4]{Sp} and \cite[3.3.1.4]{Ho}),
    \begin{align} \label{Cay}
    	\mathrm{mult}(F_{j}:\mathrm{S}^m(\mathrm{S}^d(\mathbb{C}^2))=\pi\left(d,m,\frac{1}{2}(md-j)\right)-\pi\left(d,m,\frac{1}{2}(md-j)-1\right)
    \end{align}
    where $\pi(n,k,d)$ is the number of partitions of the number $d$ into at most $k$ parts of size at most $n$. In particular, by equation (\ref{Res SS}),
    \begin{align} \label{Cay2}
	&\mathrm{mult}(F_{j}:\mathrm{Res}_{\mathfrak{s}}^{\mathfrak{sl}_n}L(\omega_k))\\
     &\hspace{1 cm}=\pi\left(n-k,k,\frac{1}{2}(k(n-k)-j)\right)-\pi\left(n-k,k,\frac{1}{2}(k(n-k)-j)-1\right) \nonumber
    \end{align}
    One has a recurrence relation
    $$\pi(n,k,d)=\pi(n-1,k,d)+\pi(n,k-1,d-n).$$
    There is also a generating function of $\pi(n,k,d)$ (cf. \cite{Ga, Sy}),
    $$\dfrac{G_n(t)G_k(t)}{G_{n+k}(t)}=\sum_{d=0}^\infty \pi(n,k,d) t^d,\hspace{0.25 cm} G_m(t)=\prod_{i=1}^m \dfrac{1}{1-t^i}.$$
    

\begin{rmk}
	It can be checked that there is a bijection from $P_k^n(d)$ to the set of partitions of the number $d-k(k+1)/2$ of length at most $k$ and size at most $n$ given by $a_i\mapsto b_i:=a_i-(k-i+1)$. Hence
	$$p_k^n(d)=\pi\left(n-k,k,d-k(k+1)/2\right).$$
	Then
	\begin{align*}
		p_k^n\left(\frac{kn-j+k}{2}\right)&= \pi\left(n-k,k,\frac{kn-j+k}{2}-\frac{k(k+1)}{2}\right)\\
		&=\pi\left(n-k,k,\frac{1}{2}(k(n-k)-j)\right).
	\end{align*}

\end{rmk}

\subsection{Non-principal subalgebras}\label{subsection: non-principal}
    
    
    
    Let $\mathfrak{s}$ be an $\mathfrak{sl}_2$ subalgebra of type $[d_1, d_2,\dots, d_m]$. For each $k=1, 2, \dots, n-1$, let $\Lambda_k^{[d_1, d_2,\dots, d_m]}$ be the multiset of weights in $\mathrm{Res}_\mathfrak{s}^{\mathfrak{sl}_n} L(\omega_k)$. If there is no ambiguity, we simply write $\Lambda_k$. Consider $H_{[d_1,d_2,\ldots,d_m]}$ acting on the natural basis $\{e_{i_1}\wedge e_{i_2}\wedge \cdots \wedge e_{i_k}\mid 1\leq i_1 < i_2 <\cdots < i_k\leq n\}$ of $L(\omega_k)\cong \wedge^k (\mathbb{C}^n)$. Label the diagonal entries of $H_{[d_1,d_2,\ldots,d_m]}$ as $h_{ii}$ for $i=1,\ldots,n$. Then 
    $$\Lambda_k=\{h_{i_1 i_1}+h_{i_2 i_2}+\cdots h_{i_k i_k}\mid 1\leq i_1 < i_2 < \cdots < i_k \leq n\}.$$
    In fact, $\mathrm{Res}_\mathfrak{s}^{\mathfrak{sl}_n} L(\omega_k)$ is completely described by $\Lambda_k$.  

\begin{rmk}
    Let $\mathfrak{s}$ be a principal $\mathfrak{sl}_2$ subalgebra. As representations of $\mathfrak{s}$, $\mathbb{C}^n \cong F_d \cong \mathrm{S}^d(\mathbb{C}^2)$ where $d+1=n$. Then
	$$L(\omega_k)\cong_{\mathfrak{sl}_n} \wedge^k (\mathbb{C}^n) \cong_{\mathfrak{s}} \wedge^k (F_d) \cong_{\mathfrak{s}} \wedge^k(\mathrm{S}^d(\mathbb{C}^2)).$$
	By \cite{LW}, the $q$-binomial coefficient of $\wedge^k(\mathrm{S}^d(\mathbb{C}^2))$ and $\mathrm{S}^k(\mathrm{S}^{n-k}(\mathbb{C}^2))$ are equal:
	$$\binom{d+1}{k}_q=\binom{k+(n-k)}{k}_q.$$
	Hence
	$$L(\omega_k) \cong_{\mathfrak{s}} \wedge^k(\mathrm{S}^d(\mathbb{C}^2)) \cong_{\mathfrak{s}} \mathrm{S}^k(\mathrm{S}^{n-k}(\mathbb{C}^2)).$$
	By the Hermite reciprocity law, $\mathrm{Res}_{\mathfrak{s}}^{\mathfrak{sl}_n} L(\omega_k) \cong \mathrm{Res}_{\mathfrak{s}}^{\mathfrak{sl}_n} L(\omega_{n-k})$. 
    Since
	$$\binom{n}{i}_q\neq \binom{n}{j}_q$$
	for any distinct $i,j\in \{1, 2, \dots,\left\lfloor \frac{n}{2} \right\rfloor\}$, we have $\mathrm{Res}_{\mathfrak{s}}^{\mathfrak{sl}_n} L(\omega_i) \ncong \mathrm{Res}_{\mathfrak{s}}^{\mathfrak{sl}_n} L(\omega_{j})$ for any distinct $i,j\in \{1, 2, \dots,\left\lfloor \frac{n}{2} \right\rfloor\}.$ This implies that we can consider $\mathrm{Res}_{\mathfrak{s}}^{\mathfrak{sl}_n} L(\omega_k)$ for only $k$ up to $\left\lfloor \frac{n-1}{2} \right\rfloor$.
\end{rmk}

\begin{rmk}
    In the view of representation theory of $\mathfrak{sl}_n$. a representation $V$ of $\mathfrak{sl}_n$ and its dual $V^*$ are not isomorphic in general. However, a representations of $\mathfrak{sl}_2$ are self dual, it immediately follows that 
    $$\mathrm{Res}_{\mathfrak{s}}^{\mathfrak{sl}_n} V\cong \mathrm{Res}_{\mathfrak{s}}^{\mathfrak{sl}_n} V^*$$
    for any $\mathfrak{sl}_2$ subalgebra $\mathfrak{s}$.
\end{rmk}

    According to our knowledge, it is very complicated to explicitly determine $\Lambda_k$ or multiplicities for arbitrary $k$ and partitions. We investigate $\mathrm{Res}_{\mathfrak{s}}^{\mathfrak{sl}_n} L(\omega_k)$ for small $k$ as follows. For $k=1$, $\Lambda_1$ precisely consists of the diagonal entries of  $H_{[d_1,d_2,\ldots,d_m]}$. It immediately follows that
    $$\mathrm{Res}_{\mathfrak{s}_{[d_1,d_2,\ldots,d_m]}}^{\mathfrak{sl}_n} L(\omega_1)\cong F_{d_1-1}\oplus F_{d_2-1}\oplus \cdots \oplus F_{d_m-1}.$$
    
    The following proposition describes the case $k=2$.
    
    \medskip
    
\begin{prop}\label{non principal, k=2}
    Let $\mathfrak{s}$ be an $\mathfrak{sl}_2$-subalgebra of $\mathfrak{sl}_n$ of type $[d_1,d_2,\dots,d_m]$. Then
    $$\mathrm{Res}_\mathfrak{s}^{\mathfrak{sl}_n} L(\omega_2) \ \cong \ \left(\bigoplus_{d_i\geq 3}\mathrm{Res}_{\mathfrak{s}_{[d_i]}}^{\mathfrak{sl}_{d_i}}L(\omega_2)\right)\oplus sF_0\oplus \left(\bigoplus_{i<j} F_{d_i-1}\otimes F_{d_j-1}\right)$$
    where $s$ is the number of $d_i$ such that $d_i=2$.
\end{prop}
\begin{proof}

    
    For $h_{ii},h_{jj}\in\{d_s-1,d_s-3,...,-(d_s-3),-(d_s-1)\}$ where $d_s\geq 3$, the weights $h_{ii}+h_{jj}$ contribute $\mathrm{Res}_{\mathfrak{s}_{[d_s]}}^{\mathfrak{sl}_{d_s}}L(\omega_2)$. If $d_s=2$, then $h_{ii}+h_{jj}=0$ and so it contributes $F_0$.
    
    For $h_{ii}\in\{d_s-1,d_s-3,...,-(d_s-3),-(d_s-1)\}$ and $h_{jj}\in\{d_t-1,d_t-3,...,-(d_t-3),-(d_t-1)\}$ where $s<t$, the weights $h_{ii}+h_{jj}$ contribute
    $$F_{d_s-d_t}\oplus F_{d_s-d_t+2}\oplus \cdots \oplus F_{d_s+d_t-2}$$
    which is isomorphic as $\mathfrak{s}$-representations to $F_{d_s-1}\otimes F_{d_t-1}$.
\end{proof}

    Note that $\mathrm{Res}_\mathfrak{s}^{\mathfrak{sl}_n} L(\omega_2)$ can be expressed in terms of branching rules to principal subalgebras which already determined in subsection \ref{subsection: principal}. 
    Unfortunately, we are not able to do so for $k\geq 3$ and arbitrary partitions. We believe that to solve this problem, other techniques should be involved. However, if $k$ and a partition $[d_1, d_2, \dots d_m]$ are given, one can explicitly determine $\Lambda_k^{[d_1, d_2, \dots d_m]}$ as shown in the following example.

\begin{ex}\label{example: [4,3]}
    Let $\mathfrak{s}$ be an $\mathfrak{sl}_2$-subalgebra of $\mathfrak{sl}_{7}$ of type $[4,3]$. The diagonal entries of $H_{[4,3]}$ are 3, 1, -1, -3, 2, 0, -2. Then the multiset $\Lambda_3$ consists of 3, 1, 6, 4, 2, -1, 4, 2, 0, 2, 0, -2, 5, 3, 1, -3, 2, 0, -2, 0, -2, -4, 3, 1, -1, 
    -2, -4, -6, 1, -1, -3, -1, -3, -5, 0.
    
\noindent Then we rearrange these weights as
\smallskip

\begin{center}
\begin{tabular}{cccccccccccccc}
6 & & 4 & & 2 & & 0 & & -2 & & -4 & & -6\\
&5 & & 3 & & 1 & &-1 & & -3 & & -5 & \\
&&4 & & 2 & & 0 & & -2 & & -4\\
&&&3 & & 1 & & -1 & & -3\\
&&&3 & & 1 & & -1 & & -3\\
&&&&2 & & 0 & & -2\\
&&&&2 & & 0 & & -2\\
&&&&&1 & & -1\\
&&&&&&0\\

\end{tabular}
\end{center}
\smallskip
where each row represents irreducible representation of $\mathfrak{sl}_2$ of highest weight equal to the left most number. Hence,

$$\mathrm{Res}_\mathfrak{s}^{\mathfrak{sl}_7} L(\omega_3)\cong F_0 \oplus F_1 \oplus 2F_2 \oplus 2F_3 \oplus F_4 \oplus F_5 \oplus F_6.$$
\end{ex}

    
    Previously, we fix a small $k$ and consider arbitrary partitions. On the other hand, we consider a fixed partition and arbitrary $k$. 
    In particular, we consider partitions of the type $[r,1^{n-r}]$ where $1^{n-r}$ means there are $n-r$ copies of $1$'s, and the type $[r, s]$. By convention, $L(\omega_0)$ means a one dimensional representation.

\medskip

\begin{prop}\label{prop: r11111}
	Let $k=1,2,...,\left\lfloor \frac{n}{2} \right\rfloor$. For any integer $r\geq 1$,
	\begin{align*}
		\mathrm{Res}_{\mathfrak{s}_{[r,1^{n-r}]}}^{\mathfrak{sl}_n}L(\omega_k) \ \cong \  \bigoplus_{j=k-\alpha}^{\beta} \binom{n-r}{j}\mathrm{Res}_{\mathfrak{s}_{[r]}}^{\mathfrak{sl}_{r}}L(\omega_{k-j})\oplus \binom{n-r}{k}F_0,
	\end{align*}
	where $\alpha=\min\{k,r\}$, $\beta=\min\{k-1,n-r\}$ and $\binom{n-r}{k}=0$ if $n-r<k$.
\end{prop}

\begin{proof}
	The diagonal entries of $H_{[r,1^{n-r}]}$ are $r-1, r-3, \dots, -(r-1), 0, 0, \dots, 0$ with $n-r$ zeroes. Let $\lambda=a_1+a_2+\cdots +a_k\in\Lambda_k^{[r,1^{n-r}]}$. If all of $a_i$'s belong to $\{0, 0, \dots, 0\}$, they are all zeroes, which contribute $\binom{n-r}{k}F_0$. On the other hand, suppose that at least one of $a_i$'s is chosen from $r-1, r-3, \dots, -(r-1)$. The largest amount of $a_i$'s in $\{r-1, r-3, \dots, -(r-1)\}$ is $\alpha$ and $(k-\alpha)$ $a_i$'s are collected from $\{0, 0, \dots, 0\}$. Also, the least amount of $a_i$ in $\{r-1, r-3, \dots, -(r-1)\}$ is one and the rest, $(k - 1)$ $a_i$'s, are collected from $\{0, 0, \dots, 0\}$. So the largest amount of $a_i$'s in $\{0, 0, \dots, 0\}$ is $\beta$. Thus we get $\displaystyle\bigoplus_{j=k-\alpha}^{\beta} \binom{n-r}{j}\mathrm{Res}_{\mathfrak{s}_{[r]}}^{\mathfrak{sl}_{r}}L(\omega_{k-j})$ from this case.
\end{proof}

\medskip

\begin{prop}\label{prop: rs}
    Let $k=1,2,...,\left\lfloor \frac{n}{2} \right\rfloor$. For any integers $r\geq s \geq 1$,
    \begin{align*}
		\mathrm{Res}_{\mathfrak{s}_{[r, s]}}^{\mathfrak{sl}_n}L(\omega_k)\cong \bigoplus_{j=0}^{\beta} \left(\mathrm{Res}_{\mathfrak{s}_{[r]}}^{\mathfrak{sl}_{r}}L(\omega_{k-j})\otimes \mathrm{Res}_{\mathfrak{s}_{[s]}}^{\mathfrak{sl}_{s}}L(\omega_{j})\right),
	\end{align*}
	where $\beta=\min\{k,s\}$ and we regard $\mathrm{Res}_{\mathfrak{s}_{[s]}}^{\mathfrak{sl}_{s}}L(\omega_{s})=F_0$.
\end{prop}
\begin{proof}
    We prove by considering $\Lambda_k^{[r,s]}$. Let $A=\{r-1, r-3, \dots, -(r-1)\}$ and $B=\{s-1, s-3, \dots, -(s-1)\}$. Each element of $\Lambda_k^{[r,s]}$ is coming from choosing $k-j$ elements from $A$ and $j$ elements from $B$ and sum them. In other words, $\Lambda_k^{[r,s]}$ is a disjoint union of the multiset $B_j$ for $j=0, 1, \dots, \beta$ where
    $$B_j=\{a_1+ \cdots + a_{k-j}+ b_1+\cdots +b_j\mid a_i\in A, \text{$a_i$'s distinct}, b_i\in B, \text{$b_i$'s distinct}\},$$
    and $a_0$ (resp. $b_0$) means no elements from $A$ (resp. $B$). Note that $B_j$ precisely consists of weights coming from $\mathrm{Res}_{\mathfrak{s}_{[r]}}^{\mathfrak{sl}_{r}}L(\omega_{k-j})\otimes \mathrm{Res}_{\mathfrak{s}_{[s]}}^{\mathfrak{sl}_{s}}L(\omega_{j})$. This completes the proof.
\end{proof}

\begin{ex}
    We compute $\mathrm{Res}_{\mathfrak{s}_{[4, 3]}}^{\mathfrak{sl}_7}L(\omega_3)$ by using Proposition \ref{prop: rs} and compare to the Example \ref{example: [4,3]} that the result are the same.  
    \begin{align*}
        \mathrm{Res}_{\mathfrak{s}_{[4, 3]}}^{\mathfrak{sl}_7}L(\omega_3)
        &\cong\left( \mathrm{Res}_{\mathfrak{s}_{[4]}}^{\mathfrak{sl}_4}L(\omega_3)\otimes \mathrm{Res}_{\mathfrak{s}_{[3]}}^{\mathfrak{sl}_3}L(\omega_0)\right)\oplus \left( \mathrm{Res}_{\mathfrak{s}_{[4]}}^{\mathfrak{sl}_4}L(\omega_2)\otimes \mathrm{Res}_{\mathfrak{s}_{[3]}}^{\mathfrak{sl}_3}L(\omega_1)\right)\\ &\hspace{0.5 cm}\oplus \left( \mathrm{Res}_{\mathfrak{s}_{[4]}}^{\mathfrak{sl}_4}L(\omega_1)\otimes \mathrm{Res}_{\mathfrak{s}_{[3]}}^{\mathfrak{sl}_3}L(\omega_2)\right) \oplus \left( \mathrm{Res}_{\mathfrak{s}_{[4]}}^{\mathfrak{sl}_4}L(\omega_0)\otimes \mathrm{Res}_{\mathfrak{s}_{[3]}}^{\mathfrak{sl}_3}L(\omega_3)\right)\\
        &\cong(F_3 \otimes F_0) \oplus ((F_0\oplus F_4)\otimes F_2)\oplus (F_3\otimes F_2) \oplus (F_0 \otimes F_0)\\
        &\cong F_3\oplus (2F_2\oplus F_4 \oplus F_6)\oplus (F_1\oplus F_3 \oplus F_5)\oplus F_0\\
        &\cong F_0 \oplus F_1 \oplus 2F_2 \oplus 2F_3 \oplus F_4 \oplus F_5 \oplus F_6.
    \end{align*}
\end{ex}

\section{Examples}\label{section: example coro}
The following examples show how to apply Theorem \ref{thm: general recursion formula} repeatedly to determine multiplicities.

\begin{ex}\label{example: principal}
    Let $\mathfrak{s}$ be a principal $\mathfrak{sl}_2$ subalgebra of $\mathfrak{sl}_5$.  For each $\lambda\in \Lambda^+$, let $m_j(\lambda)=\mathrm{mult}(F_{j}: \mathrm{Res}_{\mathfrak{s}}^{\mathfrak{sl}_n}L(\lambda))$ and $\omega_i$ defined as in Section \ref{section: recursion}. We want to compute $m_0(2\omega_1+\omega_3)$. Observe that $P(2\omega_1,3)=\{2\omega_1+\omega_3, \omega_1+\omega_4\}$ as shown in the figure below: 
    \begin{center}
\ytableausetup{centertableaux}
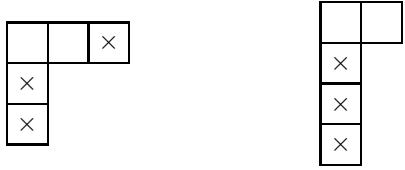
\begin{ytableau}
\none & & & \times\\
\none & \times\\
\none & \times
\end{ytableau}
\hspace{50pt}
\begin{ytableau}
\none & & \\
\none & \times\\
\none & \times\\
\none & \times
\end{ytableau}
    \end{center}
    By Theorem \ref{thm: general recursion formula}, we get
    $$m_0(2\omega_1+\omega_3)=\sum_{j=0}^\infty m_j(2\omega_1) m_j(\omega_3)-m_0(\omega_1+\omega_4).$$
    So, we need to compute $m_0(\omega_1+\omega_4)$, $m_d(2\omega_1)$, and $m_d(\omega_3)$ for $d\in \mathbb{Z}_{\geq 0}$. Note that $P(\omega_1,4)=\{\omega_1+\omega_4, 0\}$ as the following figure. 
    \begin{center}
\ytableausetup{centertableaux}
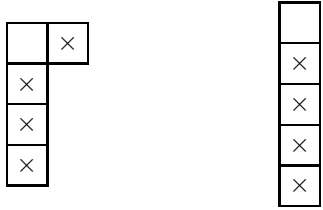
\begin{ytableau}
\none & & \times\\
\none & \times\\
\none & \times\\
\none & \times
\end{ytableau}
\hspace{50pt}
\begin{ytableau}
\none & \\
\none & \times\\
\none & \times\\
\none & \times\\
\none & \times
\end{ytableau}
    \end{center}  
    Again, by Theorem \ref{thm: general recursion formula},
    $$m_0(\omega_1+\omega_4)=\sum_{j=0}^\infty m_j(\omega_1) m_j(\omega_4)-m_0(0)=m_4(\omega_1) m_4(\omega_4)-m_0(0)=1-1=0$$
    where the second equation follows from $\mathrm{Res}_\mathfrak{s}^{\mathfrak{sl}_5} L(\omega_1)\cong \mathrm{Res}_\mathfrak{s}^{\mathfrak{sl}_5} L(\omega_4)\cong F_4$.
    For $d\in\mathbb{Z}_{\geq 0}$,
    \begin{align*}
        m_d(2\omega_1)&=\sum_{(j,j')} m_j(\omega_1)m_{j'}(\omega_1)-\sum_{\mu\in P(\omega_1, 1)-\{2\omega_1\}}m_d(\mu)\\
        &=\sum_{(j,j')} m_j(\omega_1)m_{j'}(\omega_1)-m_d(\omega_2).
    \end{align*}    
    Observe that the summand $m_j(\omega_1)m_{j'}(\omega_1)=1$ if and only if $j=j'=4$. Since $\Lambda_2^{[5]}=\{6, 4, 2, 0, 2, 0, -2, -2, -4, -6\}$,  $\mathrm{Res}_\mathfrak{s}^{\mathfrak{sl}_5} L(\omega_2)\cong F_2\oplus F_6$. Then
    \[m_d(2\omega_1)=
        \begin{cases}
            1   & \quad \text{if } d=0, 4\text{ or } 8,\\
            0   & \quad \text{otherwise.}
        \end{cases}
    \]
    Since $\Lambda_3^{[5]}=\{6, 4, 2, 0, 2, 0, -2, -2, -4, -6\}$,  $\mathrm{Res}_\mathfrak{s}^{\mathfrak{sl}_5} L(\omega_3)\cong F_2\oplus F_6$. Hence
    $$m_0(2\omega_1+\omega_3)=\sum_{j=0}^\infty m_j(2\omega_1) m_j(\omega_3)-m_0(\omega_1+\omega_4)=0.$$
\end{ex}

\begin{ex}
    Retain all notations in Example \ref{example: principal} except that we consider the multiplicity $m_1(2\omega_1 + \omega_3)$ in the branching rule of $L(2\omega_1 + \omega_3)$ to an $\mathfrak{sl}_2$ subalgebra $\mathfrak{s}$ of type $[3,2]$. We still need to compute $m_1(\omega_1+\omega_4)$, $m_d(2\omega_1)$, and $m_d(\omega_3)$ for $d\in \mathbb{Z}_{\geq 0}$. 
    It is obvious that
    $$\mathrm{Res}_\mathfrak{s}^{\mathfrak{sl}_5} L(\omega_1)\cong \mathrm{Res}_\mathfrak{s}^{\mathfrak{sl}_5} L(\omega_4)\cong F_1 \oplus F_2.$$
    Since $\Lambda_1^{[3,2]}=\Lambda_4^{[3,2]}=\{2, 0, 3, 1, -2, 1, -1, -1, -3, 0\}$,
    $$\mathrm{Res}_\mathfrak{s}^{\mathfrak{sl}_5} L(\omega_2)\cong \mathrm{Res}_\mathfrak{s}^{\mathfrak{sl}_5} L(\omega_3)\cong F_0\oplus F_1 \oplus F_2 \oplus F_3.$$
    By Theorem \ref{thm: general recursion formula}, we get
    \begin{align*}
        m_1(\omega_1+\omega_4)&=\sum_{j, j'} m_j(\omega_1) m_j(\omega_4)-m_1(0)\\
        &=m_1(\omega_1)m_2(\omega_4)+m_2(\omega_1)m_1(\omega_4)-m_1(0) \\
        &=1+1-0 n\\
        &=2.
    \end{align*}
    Next, we consider 
    $$m_d(2\omega_1)=\sum_{(j,j')} m_j(\omega_1)m_{j'}(\omega_1)-m_d(\omega_2)$$ 
    for each $d$. Then
    \begin{align*}
        m_0(2\omega_1)&=m_1(\omega_1)m_1(\omega_1)+m_2(\omega_1)m_2(\omega_1)-m_0(\omega_2)=2-1=1\\
        m_1(2\omega_1)&=m_1(\omega_1)m_2(\omega_1)+m_2(\omega_1)m_1(\omega_1)-m_1(\omega_2)=2-1=1\\
        m_2(2\omega_1)&=m_1(\omega_1)m_1(\omega_1)+m_2(\omega_1)m_2(\omega_1)-m_2(\omega_2)=2-1=1\\
        m_3(2\omega_1)&=m_1(\omega_1)m_2(\omega_1)+m_2(\omega_1)m_1(\omega_1)-m_3(\omega_2)=2-1=1\\
        m_4(2\omega_1)&=m_2(\omega_1)m_2(\omega_1)-m_4(\omega_2)=1-0=1\\
        m_d(2\omega_1)&=0 \text{ for } d\geq 5.
    \end{align*}
    Hence
    \begin{align*}
        m_1(2\omega_1+\omega_3)&=m_0(2\omega_1)m_1(\omega_3)+m_1(2\omega_1)m_0(\omega_3)+m_1(2\omega_1)m_2(\omega_3)\\
        & \hspace{0.3 cm} +m_2(2\omega_1)m_1(\omega_3)+m_2(2\omega_1)m_3(\omega_3)+m_3(2\omega_1)m_2(\omega_3)\\
        & \hspace{0.3 cm} +m_4(2\omega_1)m_3(\omega_3)-m_1(\omega_1+\omega_4)\\
        &=7-2=5.
    \end{align*}
\end{ex}

\medskip

\section*{Acknowledgements}
    The first author would like to thank Department of Mathematics, National University of Singapore for an excellent working environment, where a part of this work was carried out. The first author is also supported by the MOE grant A-0004586-00-00. 
    The third author acknowledges the support 
      by Office of the Permanent Secretary, Ministry of Higher Education, Science, Research and Innovation  (OPS MHESI), Thailand Science Research and Innovation (TSRI) (Grant No. RGNS 64-096).

\end{document}